\documentclass[11pt,letterpaper]{amsart}
\usepackage[plainpages=false]{hyperref}
\usepackage{amsfonts,latexsym,rawfonts,amsmath,amssymb,amsthm, mathrsfs, lscape, calligra}
\usepackage{verbatim}
\usepackage{enumerate}
\usepackage{centernot}
\usepackage{mathtools}
\usepackage{tikz}
\usepackage{tikz-cd}

\usepackage[all]{xy}
\usepackage{graphicx,psfrag}

\usepackage{array, tabularx, color}

\usepackage{setspace}

\newtheorem{thm}{Theorem}[section]
\newtheorem{cor}[thm]{Corollary}
\newtheorem{lem}[thm]{Lemma}

\newtheorem{prop}[thm]{Proposition}
\newtheorem{prob}[thm]{Question}

\newtheorem{exam}[thm]{Example}

\theoremstyle{remark}
\newtheorem{rmk}[thm]{Remark}

\theoremstyle{definition}
\newtheorem{defi}[thm]{Definition}

\newcommand{\CBbb}{\mathbb C}

\newcommand{\PBbb}{\mathbb P}

\newcommand{\RBbb}{\mathbb R}

\newcommand{\ZBbb}{\mathbb Z}

\newcommand{\Ecal}{\mathcal E}
\newcommand{\Fcal}{\mathcal F}

\newcommand{\Ical}{\mathcal I}

\newcommand{\Lcal}{\mathcal L}

\newcommand{\Ocal}{\mathcal O}

\newcommand{\Qcal}{\mathcal Q}
\newcommand{\Rcal}{\mathcal R}

\newcommand{\Tcal}{\mathcal T}

\newcommand{\Xcal}{\mathcal X}

\DeclareMathOperator{\Ext}{Ext}

\DeclareMathOperator{\Tor}{Tor}

\numberwithin{equation}{section}

\makeatletter
\@namedef{subjclassname@2020}{\textup{2020} Mathematics Subject Classification}
\makeatother

\begin{document}
\title[Singularity formation]{Bubbling of rank two bundles over surfaces}
\author[Chen]{Xuemiao Chen}
\begin{abstract}
    In this paper, motivated by the singularity formation of ASD connections in gauge theory, we study an algebraic analogue of the singularity formation of families of rank two holomorphic vector bundles over surfaces. For this, we define a notion of fertile families bearing bubbles and give a characterization of it using the related discriminant. Then we study families that locally form the singularity of the type $\Ocal \oplus \Ical$ where $\Ical$ is an ideal sheaf defining points with multiplicities. We prove the existence of  fertile families bearing bubbles by using elementary modifications of the original family. As applications, we study bubble trees for a few families that form singularities of low multiplicities and use examples to give negative answers to some plausible general questions.
\end{abstract}

\thanks{}
\address{Department of pure mathematics, University of Waterloo, Ontario, Canada, N2L 3G1}\email{x67Chen@uwaterloo.ca}
\maketitle
\thispagestyle{empty}

\maketitle
\thispagestyle{empty}

\bibliographystyle{amsplain}

\tableofcontents
\section{Introduction}
Given a sequence of ASD connections with uniformly bounded $L^2$ norms on curvature over a Riemannian four manifold, the Uhlenbeck compactness (\cite{DonaldsonKronheimer:90, Uhlenbeck:82a, Uhlenbeck:82b}) states that by passing to a subsequence, up to gauge transforms, the subsequence converges to an ASD connection away from finitely many points. Those points are usually referred as the bubbling points. By rescaling properly near any of these points, this sequence converges to an ASD connection $A_\infty$ over $\RBbb^4$. Depending on the rescalings, the following could happen 
\begin{itemize}
    \item no bubbling happens at infinity of $\RBbb^4$ and $A_\infty$ is flat;
    \item no bubbling happens at infinity of $\RBbb^4$ and $A_\infty$ is not flat;
    \item bubbling happens at infinity.
\end{itemize}
Here the second case is of significant interest since then $A_\infty$ could be viewed as the bubble on the first level of a bubble tree that memorizes all the energy loss nontrivially (we will not give a formal definition for the concept of bubble trees and hope the description justifies itself in a natural way). Actually, the first case is also of interest since at least the rescalings guarantee that the multiplicity of the bubbling point is captured over $\RBbb^4$ and by rescaling faster, one might be able to see the second case again. Through the celebrated gluing construction for ASD instantons by Taubes, examples of the second case exist abundantly (\cite{Taubes82}). The bubbling phenomenon is also one of the key features in Donaldson theory (\cite{Donaldson83} \cite{DonaldsonKronheimer:90}). 

Through a conformal change, by Uhlenbeck's removable singularity theorem (\cite{Uhlenbeck:82b}), the limiting bubbling connection $A_\infty$ in turn gives an ASD connection over $S^4$ where $c_2(A_\infty)$ is less or equal to the multiplicity of the bubbling point where the equality corresponds to the first two cases above. By the celebrated ADHM construction (\cite{ADHM}), this corresponds to a special class of holomorphic vector bundles over $\PBbb^3$, usually referred as instanton bundles, which can be described by a set of linear algebra data. Donaldson later showed (\cite{Donaldson:84}) that this can be also identified with holomorphic vector bundles over $\PBbb^2=\CBbb^2 \cup \PBbb^1$ which satisfy the condition of triviality at the infinity $\PBbb^1$, i.e., it restricts to a trivial bundle over $\PBbb^1$. To be more precise, a trivialization needs to be fixed along $\PBbb^1$.

In this paper, we study an algebraic analogue of the bubbling phenomenon for families of rank two bundles over surfaces motivated by Donaldson's interpretation of the bubbling connections. More precisely, we consider a rank two reflexive sheaf $\Ecal$ over $B=\{(x,y,z)\in \CBbb^3: |x|^2+|y|^2+|z|^2<1\}$. Then the singular set of $\Ecal$ is known to consist of isolated points. Assume it contains only $0$, then $\Ecal$ can be viewed locally as a family of holomorphic vector bundles forming a point singularity at $0\in B_0=B\cap (z=0)$. A very important quantity related to the gauge theoretical study is the so-called multiplicity, which algebraically is given by
$$
k(\Ecal)= \text{ the dimension of the stalk of } \Ecal|_{B_0}^{**}/\Ecal|_{B_0} \text{ at } 0.
$$
Let $p: \widehat{B} \rightarrow B$ denote the blow-up at $0\in B$ with the exceptional divisor $\PBbb^2$. The central fiber $B_0$ naturally selects out the line at infinity $\PBbb^1\subset \PBbb^2$. Then we can look at all the possible extensions of $(p^*\Ecal)|_{\widehat{B}\setminus \PBbb^2}$ across $\PBbb^2$. Motivated by the interpretations of the ASD instantons by Donaldson, we want to find extensions $\widehat{\Ecal}$ of $(p^*\Ecal)|_{\widehat{B}\setminus \PBbb^2}$ across the exceptional divisor so that 
$$
\widehat{\Ecal}|_{\PBbb^1}\cong \Ocal_{\PBbb^1}^{\oplus 2}
$$
and to give nontrivial bubbling, we further require that 
$$
\widehat{\Ecal}|_{\PBbb^2}^{**} \ncong \Ocal_{\PBbb^2}^{\oplus 2}.
$$
If such extensions exist, we call the given families \emph{fertile} and $\widehat{\Ecal}|_{\PBbb^2}$ the \emph{bubble} of the family. Corresponding to the analytic picture above, if only the first condition is satisfied, we call them \emph{cone} families. We also refer the fertile and cone familie together as non-barren families. If there exists no extensions which restricts to the trivial bundle over $\PBbb^1$, the families will be referred as \emph{barren}. For a fertile family, by definition, the bubble $\widehat{\Ecal}|_{\PBbb^2}$ is semistable. Thus it is related to the notion of optimal extensions studied in a more general context in \cite{ChenSun:20b}. In particular, if such a semistable extension exists, it must be given by the optimal extension obtained in \cite{ChenSun:20b}. However, examples do exist where the optimal extensions do not satisfy the triviality at infinity, i.e., $\widehat{\Ecal}|_{\PBbb^1}\ncong \Ocal_{\PBbb^1}^{\oplus 2}$ (See Proposition \ref{BarrenFamily}). This is natural in the sense that as in the analytic case, to obtain bubbling, one needs to choose rescalings carefully. The same is expected to be true in the algebraic setting where this simply means the fixed family is not rescaled correctly and we need to ``rescale" the family properly to gain fertile families forming the same singularity. Another key difficulty lies in the triviality condition over $\PBbb^1$ which is a restriction on codimension two subset. Namely, for reflexive sheaves, if we know two such sheaves are isomorphic outside some codimension two subvariety, then they must be isomorphic, thus leaving no room for modifications if we want to stay in the realm of reflexive sheaves, which is unlike the modifications along a negative divisor used in \cite{ChenSun:20b}. Also, the restriction of a reflexive sheaf to a codimension two subvariety could behave very badly, for example, it could carry torsion unlike the case of restriction to a smooth hypersurface where we always get a torsion free sheaf. Those are two of the main technical points that we need to deal with in order to obtain a fertile family from a given family.

Now we state the main results of this paper.
\subsection*{Main results} Below we denote the discriminant of a torsion free sheaf $\underline{\Fcal}$ over $\PBbb^2$ as 
$$
\Delta(\underline{\Fcal})=4c_2(\underline{\Fcal})-c_1(\underline{\Fcal})^2 \in \ZBbb.
$$
We first give a characterization for when a given extension has nice geometric properties such as the restriction being semistable and triviality at infinity (see Section \ref{Bubble}). 

\begin{thm}\label{Thm1.1}
    Given a rank two reflexive sheaf $\Ecal$ over $B$, then 
    \begin{enumerate}
         \item an extension $\widehat{\Ecal}$ of $\Ecal$ is semistable, i.e., $\widehat{\Ecal}|_{\PBbb^2}$ is semistable, if and only if $\Delta(\widehat{\Ecal}|_{\PBbb^2})$ is the largest among all the reflexive extensions;
         
        \item any reflexive extension $\widehat{\Ecal}$ of $\Ecal$ satisfies 
    $$
    \frac{\Delta(\widehat{\Ecal}|_{\PBbb^2})}{4}\leq k(\Ecal);
    $$  
      
       \item 
    an extension $\widehat{\Ecal}$ satisfies 
    $$
    \widehat{\Ecal}|_{\PBbb^1}\cong \Ocal_{\PBbb^1}^{\oplus 2}
    $$
    if and only if 
    $$
    \frac{\Delta(\widehat{\Ecal}|_{\PBbb^2})}{4}=k(\Ecal).
    $$
    \end{enumerate}
   In particular, a family is fertile if and only if $\Delta(\widehat{\Ecal}|_{\PBbb^2})=k(\Ecal)$ and 
    $$
    \widehat{\Ecal}|_{\PBbb^2}^{**} \ncong \Ocal_{\PBbb^2}^{\oplus 2}.
    $$
\end{thm}

\begin{rmk}
This in particular gives an algebraic analogue of no loss of energy at infinity in the analytic case above that the nonbarren families can be characterized by the condition of triviality at infinity. 
\end{rmk}

Then we make progress in understanding the bubbling phenomenon by studying families $\Ecal$ forming singularities of the type $\Ocal \oplus \Ical$. More precisely, we assume 
$$
\Ecal|_{B_0}\cong \Ocal_{B_0} \oplus \Ical
$$
for some ideal sheaf $\Ical$ which locally defines the origin of $B_0$ with multiplicities, i.e., $\Ocal_{B_0}/\Ical$ is supported at $0\in B_0$ and in this case, the multiplicity is given by 
$$
k(\Ecal)= \text{ the dimension of the stalk of } \Ocal_{B_0}/\Ical \text{ at } 0.
$$

The second main result of this paper is the existence of fertile families via certain natural elementary modifications of the given algebraic family which forms the singularity $\Ocal_{B_0} \oplus \Ical$.

\begin{thm}\label{Main}
    Given any family $\Ecal$ which forms the singularity $\Ocal_{B_0} \oplus \Ical$, after finitely many elementary modifications, $\Ecal$ can be transformed to a fertile family which still forms the singularity $\Ocal_{B_0} \oplus \Ical$ with the bubble being strictly semistable. More precisely, the following holds 
    \begin{enumerate}
        \item After finitely many elementary modifications along the $\Ocal_{B_0}$ factor, the families form the same singularity and will be cone families with isomorphic cones except at most one step corresponding to a fertile family; for a non-barren family, the elementary modifications along $\Ocal_{B_0}$ are cone families with the semistable extension $\widehat{\Ecal}$ satisfying 
        $$
        \widehat{\Ecal}|_{\PBbb^2} \cong \Ocal_{\PBbb^2} \oplus \underline{\Ical}
        $$
        for a fixed ideal sheaf $\underline{\Ical}$ which defines a subscheme of $\PBbb^2$ supported at $[0,0,1]$.

        \item For any cone family locally forming the singularity $\Ocal_{B_0} \oplus \Ical$, by doing elementary modifications along $\Ical$, the cone families will become fertile after finitely many steps. Furthermore, the semistable extension $\widehat{\Ecal}$ of the fertile family satisfies 
        $$
        0\rightarrow\Ocal_{\PBbb^2} \rightarrow \widehat{\Ecal}|_{\PBbb^2} \rightarrow \underline{\Ical} \rightarrow 0
        $$
        where $\underline{\Ical}$ is the same ideal sheaf as for the cone families above.
    \end{enumerate}
\end{thm}

\begin{rmk}
\begin{itemize}
    \item The nontrivial bubble found here belongs to a very special class of instantons in the ADHM construction that arises from the spectral construction associated to an ideal sheaf over $\CBbb^2$ (see \cite[Page 400]{DonaldsonKronheimer:90}).
    
    \item $\underline{\Ical}$ in general has no direct relations with $\Ical$ and could be different in nature (Corollary \ref{Different}). In particular, the cones do not necessarily recover the original singularity except for the multiplicity.

    \item The theorem says if we do elementary modification along the $\Ocal_{B_0}$ factor and then look at the extension through the blow-up, this corresponds to the analytic case that if we rescale slowly, the limit is a cone. For a cone family, if we do elementary modification along the $\Ical$ factor and look at things through the blow-up, this corresponds to rescaling faster compared to the original rescalings for a cone, then at a critical scale, we see a nontrivial bubble with no bubbling at infinity. 

    \item To solve the problem in general, it seems to the author that new modifications and new structures from the formed singularities or families are needed.
\end{itemize}
\end{rmk}

As applications, we then study the formations of a few singularities of low multiplicities. The final section consists of examples which give negative answers to a few plausible questions in general. 

\subsection*{Acknowledgment} The author would like to thank Song Sun for insightful discussions on singularity formation of Yang-Mills connections. This work is partially supported by NSERC and the ECR supplement. 

\section{Bubble}\label{Bubble}
In this section, we study a notion of \emph{bubble} associated to a family of rank two bundles forming a point singularity. Below, we let $\Ecal$ be a rank two reflexive sheaves over $B=\{(x,y,z)\in \CBbb^3: |x|^2+|y|^2+|z|^2<1\}$  where we will fix a spitting $\CBbb^3\cong \CBbb^2 \times \CBbb$ and view $\Ecal$ as a family of rank two bundles which forms a possible singularity at $0\in B_0=B \cap (\CBbb^2 \times \{0\})$. Here $\Ecal|_{B_0}$ is a torsion free sheaf and we define the multiplicity of $\Ecal$ as 
    $$
    k(\Ecal)=\textit{ the dimension of the stalk of } \Ecal|_{B_0}^{**}/\Ecal|_{B_0} \textit{ at } 0.
    $$
We also denote 
    $$
    B_z=B \cap (\CBbb^2 \times \{z\}).
    $$
\begin{rmk}
By taking different splittings $\CBbb^3 \cong \CBbb^2 \times \CBbb$ at $0\in \CBbb^3$, the definition above will in general give different multiplicities and it is actually an upper semi-continuous function of planes passing $0$ in $\CBbb^3$ (See Corollary \ref{Cor2.11}). We will use $k^g(\Ecal)$ to denote the smallest possible value of this function.
\end{rmk}

For our study of bubbles, we need to use the blow-up of $B$ at the origin 
$$
p: \widehat{B} \rightarrow B
$$
and refer the exceptional divisor as 
$$
\PBbb^2:=p^{-1}(0)
$$ 
and $\Ocal_{\widehat{B}}(1)$ as the invertible sheaf associated to $\PBbb^2$.  We also denote the strict transform of $B_0$ in $\widehat{B}$ as $\widehat{B_0}$ and 
$$
\PBbb^1:=\PBbb^2\cap \widehat{B_0}.
$$ 
Then we need to look at all the possible reflexive extensions of $(p^*\Ecal)|_{\widehat{B}\setminus \PBbb^2}$ across the exceptional divisor $\PBbb^2$. 

There exists tons of such extensions due to that the exceptional divisor has codimension one. We first define those good ones corresponding to our study of singularity formation. 
\begin{defi}
    If there exists a reflexive extension $\widehat{\Ecal}$ of a family $\Ecal$ satisfying both
    \begin{enumerate}
        \item    (triviality at infinity) $\widehat{\Ecal}|_{\PBbb^1}\cong\Ocal_{\PBbb^1}^{\oplus 2}$;
        \item (nontrivial bubble) $\widehat{\Ecal}|_{\PBbb^2}^{**}  \ncong \Ocal_{\PBbb^2}^{\oplus 2}.$
    \end{enumerate}
    then $\widehat{\Ecal}|_{\PBbb^2}$ is called a \emph{bubble} for the family $\Ecal$ and $\Ecal$ is called a \emph{fertile} family. If $\widehat{\Ecal}|_{\PBbb^2}^{**}\cong\Ocal_{\PBbb^2} \oplus \Ocal_{\PBbb^2}$ and $\widehat{\Ecal}|_{\PBbb^1}\cong \Ocal_{\PBbb^1} \oplus \Ocal_{\PBbb^1}$, we call it a \emph{cone} family and $\widehat{\Ecal}|_{\PBbb^2}$ the cone of the family. We will also refer cone families and fertile families together as nonbarren families. If the family is neither a cone family nor a fertile family, we call the family a barren family. 
\end{defi}

\begin{prop}[\cite{ChenSun:20b}]\label{BubbleUniqueness}
   There exists a unique extension  $\widehat{\Ecal}$ of a family $\Ecal$ so that 
   $\widehat{\Ecal}|_{\PBbb^2}$ is semistable of degree equal to $0$ or $1$ exclusively. In particular, a bubble for the given family, if it exists, is the unique extension $\widehat{\Ecal}$ with $\widehat{\Ecal}|_{\PBbb^2}$ being semistable. 
\end{prop}

\begin{proof}
This follows from \cite[Theorem 1.4 (II)]{ChenSun:20b} where since the sheaves considered have rank two, semistablity can be always achieved, thus it is unique since we can normalize the determinant to be trivial or $1$ here.
\end{proof}

Below we will call the reflexive extension $\widehat{\Ecal}$ of $\Ecal$ with $\widehat{\Ecal}|_{\PBbb^2}$ being semistable and $c_1\in \{0,-1\}$ the \emph{normalized} semistable extension of $\Ecal$.

Given this, the following is well-defined
\begin{defi}
    Given any family $\Ecal$, the discriminant of $\Ecal$ is defined as 
    $$
    \Delta_{\Ecal}:=\Delta(\widehat{\Ecal}|_{\PBbb^2})
    $$
    where $\widehat{\Ecal}$ is any semistable extension of $\Ecal$. 
\end{defi}

\begin{prop}\label{BarrenFamily}
    There exist barren families. 
\end{prop}

\begin{proof}
    For this, consider the family $\Ecal$ given by
    $$
    0\rightarrow \Ocal \xrightarrow{\begin{pmatrix}
        x\\
        y\\
        z\end{pmatrix}} \Ocal^{\oplus 3} \rightarrow \Ecal \rightarrow 0.
    $$
    A natural extension for the family can be given by 
    $$
    0 \rightarrow \Ocal \xrightarrow{\begin{pmatrix}
        X\\
        Y\\
        Z
        \end{pmatrix}} \Ocal_{\widehat{B}}(-1) \oplus \Ocal_{\widehat{B}}(-1) \oplus \Ocal_{\widehat{B}}(-1) \rightarrow \widehat{\Ecal} \rightarrow 0.
    $$
    Then $\widehat{\Ecal}|_{\PBbb^2}\cong \Tcal_{\PBbb^2}$ and it is stable of odd degree. By Proposition \ref{BubbleUniqueness}, this family has to be barren since the bubble always has zero degree if it exists.
\end{proof}

It turns out that the semistable extensions can also be characterized using only its discriminant.
\begin{prop}\label{Prop2.6}
$\widehat{\Ecal}|_{\PBbb^2}$ has the largest discriminant among all the extensions of $\Ecal$ if and only if $\widehat{\Ecal}|_{\PBbb^2}$ is semistable.
\end{prop}

\begin{proof}
This follows from the main result \cite{ChenSun:20b} together with the observation that the elementary modification also decreases the discriminant of the restriction of the extensions on the exceptional divisor. Suppose $\widehat{\Ecal}$ is an extension with $\widehat{\Ecal}|_{\PBbb^2}$ being unstable. Let $\underline{\Lcal}$ be the maximal destabilizing subsheaf of $\widehat{\Ecal}|_{\PBbb^2}$ which has rank one. Denote $\underline{\Qcal}=\widehat{\Ecal}|_{\PBbb^2}/\underline{\Lcal}$. The elementary modification of $\widehat{\Ecal}'$ along $\underline{\Lcal}$ is then given by 
$$
0\rightarrow \widehat{\Ecal}' \rightarrow \widehat{\Ecal} \rightarrow \iota_* \underline{\Qcal} \rightarrow 0
$$
and $\widehat{\Ecal}'|_{\PBbb^2}$ satisfies
$$
0\rightarrow \underline{\Qcal}(1) \rightarrow \widehat{\Ecal}'|_{\PBbb^2} \rightarrow \underline{\Lcal} \rightarrow 0.
$$
From this, we know 
$$
\begin{aligned}
\Delta(\widehat{\Ecal}'|_{\PBbb^2})&=\Delta(\widehat{\Ecal}|_{\PBbb^2})+2(\deg(\underline{\Lcal})-\deg(\underline{\Qcal}))-1\\
&\geq \Delta(\widehat{\Ecal}|_{\PBbb^2})+1.
\end{aligned}
$$
Then by \cite{ChenSun:20b}, if we continue in this way, it will stop at an extension with $\widehat{\Ecal}|_{\PBbb^2}$ being semistable which is unique up to tensoring with powers of $\Ocal_{\widehat{B}}(1)$. The conclusion follows. 
\end{proof}

This gives (1) of Theorem \ref{Thm1.1}.
To continue, we need the following two facts. 
\begin{lem}\label{PushImgageStabilizes}
    Given any extension $\widehat{\Ecal}$ of $\Ecal$, for $k$ large, 
    $$
    p_*(\widehat{\Ecal}(k))=p_*(\widehat{\Ecal}(k+1))=\cdots =\Ecal.
    $$
\end{lem}

\begin{proof}
It follows from the definition that 
$$
p_*(\widehat{\Ecal}(1)) \subset p_*(\widehat{\Ecal}(2))\subset \cdots \Ecal.
$$ 
By using the Noetherian property of the stalk of $\Ecal$ at $0$, we know 
$$
p_*(\widehat{\Ecal}(k))=p_*(\widehat{\Ecal}(k+1))=\cdots \subset \Ecal
$$
for $k$ larger. It remains to show that 
$$
p_*(\widehat{\Ecal}(k))\cong \Ecal.
$$
For this, it follows from the canonical map $\Ecal \rightarrow p_*p^*\Ecal$ that 
$$
\Ecal\subset  p_*(p^*\Ecal)^{**}
$$
thus $\Ecal=p_*(p^*\Ecal)^{**}$. For any $\widehat{\Ecal}$, the natural isomorphism over $\widehat{B} \setminus \PBbb^2$ can be extended as 
$$
(p^*\Ecal)^{**}\rightarrow  \widehat{\Ecal}(m) 
$$
for some $m$, which implies 
$$
\Ecal=p_*((p^*\Ecal)^{**}) \subset p_*(\widehat{\Ecal}(m)).
$$ 
The conclusion follows.
\end{proof}

\begin{lem}\label{Lem2.7}
    Given any locally free sheaf $\Fcal$ over $\widehat{B_0}$, $R^1p_*(\Fcal(-k))=0$ for $k$ large. In particular, if $\Fcal|_{\PBbb^1}\cong \Ocal_{\PBbb^1}^{\oplus r}$, then $R^1p_*\Fcal=0$.
\end{lem}

\begin{proof}
    The first statement is well-known which is essentially due to that the exceptional divisor is negative. The second statement now follows from 
    $$
    R^1p_*\Fcal(-k) \hookrightarrow R^1p_* \Fcal(-k+1)
    $$
    is surjective for any $k\geq 1$ under the given assumption. Indeed, using the natural short exact sequence 
    $$
    0 \rightarrow \Fcal(-k) \rightarrow \Fcal(-k+1) \rightarrow \Fcal(-k+1)|_{\PBbb^1} \rightarrow 0
    $$
    where $\Fcal(-k+1)|_{\PBbb^1}\cong \Ocal_{\PBbb^1}(k-1)^{\oplus 2}$ and that $H^1(\PBbb^1, \Ocal_{\PBbb^1}(k-1))=0$ for $k\geq 1$, we can obtain the claimed surjective map by pushing down the exact sequence above to be over $B$.
\end{proof}

There are useful geometric properties implied by the condition of triviality at infinity.
\begin{lem}\label{Thm1.5}
    Suppose $\widehat{\Ecal}|_{\PBbb^1}\cong\Ocal_{\PBbb^1} \oplus \Ocal_{\PBbb^1}$, then 
    $$p_*\widehat{\Ecal}\cong \Ecal$$ 
    and 
    $$
     k(\Ecal)=\frac{\Delta_{\Ecal}}{4}.
    $$
\end{lem}

\begin{proof}
   We first prove $p_*\widehat{\Ecal}\cong \Ecal$. By Lemma \ref{PushImgageStabilizes}, for $k$ large 
    $$
    p_*(\widehat{\Ecal}(k))=p_*(\widehat{\Ecal}(k+1))=\cdots \Ecal,
    $$
    it suffices to show that 
    $$
    p_*(\widehat{\Ecal})=p_*(\widehat{\Ecal}(k))
    $$
    for any $k\geq 0$. For this, by pushing forward the following exact sequence
    $$
0\rightarrow \widehat{\Ecal}(k)\rightarrow \widehat{\Ecal}(k+1) \rightarrow \widehat{\Ecal}|_{\PBbb^2}\otimes \Ocal_{\PBbb^2}(-k-1) \rightarrow 0
    $$
    to be over $B$, we have 
       $$
0\rightarrow p_*(\widehat{\Ecal}(k))\rightarrow p_*(\widehat{\Ecal}(k+1)) \rightarrow H^0(\PBbb^2, \widehat{\Ecal}|_{\PBbb^2}\otimes \Ocal_{\PBbb^2}(-k-1))=0
    $$
    where the last equality follows from that $\widehat{\Ecal}|_{\PBbb^2}$ is semistable of degree zero, which forces the first map to be an isomorphism. The claimed equality above follows from induction. It remains to show $$\Delta(\widehat{\Ecal}|_{\PBbb^2})=4 k(\Ecal).$$ 
    For this, push down the following exact sequence 
    $$
    0 \rightarrow \widehat{\Ecal}(-\widehat{B_0}) \rightarrow \widehat{\Ecal} \rightarrow \widehat{\Ecal}|_{\widehat{B_0}} \rightarrow 0
    $$
    to be over $B$, and using $\widehat{\Ecal}(-\widehat{B_0}) \cong \widehat{\Ecal}(1)$, we get 
    $$
    \begin{aligned}
    &0\rightarrow \Ecal \xrightarrow{z} \Ecal \rightarrow \Ocal_{B_0}^{\oplus 2} \\
    &\rightarrow R^1p_*(\widehat{\Ecal}(1)) \rightarrow R^1p_*(\widehat{\Ecal}) \rightarrow R^1p_*(\widehat{\Ecal}|_{\widehat{B_0}})\\
    &\rightarrow R^2p_*(\widehat{\Ecal}(1)) \rightarrow R^2p_*(\widehat{\Ecal}) \rightarrow R^2 p_*(\widehat{\Ecal}|_{\widehat{B_0}})\\
    &\rightarrow R^3 p_*(\widehat{\Ecal}(1)) \rightarrow R^3 p_*(\widehat{\Ecal}) \rightarrow R^3 p_*(\widehat{\Ecal}|_{\widehat{B_0}}) \rightarrow 0
    \end{aligned}
    $$
    where the first row gives
    $$
    0\rightarrow \Ecal|_{B_0} \rightarrow \Ocal_{B_0}^{\oplus 2}=(\Ecal|_{B_0})^{**}.
    $$
    This implies
    $$
    \begin{aligned}
    k(\Ecal)&=\chi'(R^\bullet p_* \widehat{\Ecal})-\chi'(R^\bullet p_* (\widehat{\Ecal}(1)))-\chi'(R^\bullet p_* (\widehat{\Ecal}|_{\widehat{B}_0}))\\
    &=\chi'(\widehat{\Ecal}|_{\PBbb^2}(-1))\\
    &=\Xcal(\widehat{\Ecal}|_{\PBbb^2}(-1))\\
    &=c_2(\widehat{\Ecal}|_{\PBbb^2}).
    \end{aligned}
    $$  
    Here 
    $$
    \chi'(R^\bullet p_* \widehat{\Fcal}):=\sum_{i\geq 1} (-1)^i \dim_{\CBbb} R^ip_*(\widehat{\Fcal})_0;
    $$ 
    the second equality follows from $R^\bullet p_*(\widehat{\Ecal}|_{\widehat{B_0}})$ having zero stalk at $0$ for $p\geq 1$ by Lemma \ref{Lem2.7}; the third equality follows from $H^0(\PBbb^2, \widehat{\Ecal}|_{\PBbb^2}(-1))=0$ since $\widehat{\Ecal}|_{\PBbb^2}$ is semistable of degree $0$; the last equality follows from a straightforward computation using the Riemann-Roch theorem for sheaves over surface.
\end{proof}

The computation above gives more in general.
\begin{cor}\label{computation}
    If 
$$
\widehat{\Ecal}|_{\PBbb^1}=\Ocal_{\PBbb^1} \oplus \Ocal_{\PBbb^1}(-1)
$$
and $\widehat{\Ecal}|_{\PBbb^2}$ is stable, then 
$$
c_2(\widehat{\Ecal}|_{\PBbb^2})=k(\Ecal).
$$
In general, for the normalized semistable extension $\widehat{\Ecal}$ of $\Ecal$, the following holds
    $$
    \dim_{\CBbb} (p_*(\widehat{\Ecal}|_{\widehat{B_0}})/\Ecal|_{B_0})_0=c_2(\widehat{\Ecal}|_{\PBbb^2})-\Xcal'(R^{\bullet} p_*(\widehat{\Ecal}|_{\widehat{B_0}})) \leq k(\Ecal).
    $$
\end{cor}

Now we give a new characterization of the condition of triviality at infinity using the discriminant.

\begin{prop}\label{BubbleCharaterization}
Given an extension $\widehat{\Ecal}$, the following are equivalent 
\begin{enumerate}
    \item $\widehat{\Ecal}|_{\PBbb^1}\cong\Ocal_{\PBbb^1} \oplus \Ocal_{\PBbb^1}$;
    \item $\Delta(\widehat{\Ecal}|_{\PBbb^2})=4 k(\Ecal)$.
\end{enumerate}
\end{prop}

\begin{proof}
    The fact that (1) implies (2) follows from Lemma \ref{Thm1.5}. To show that (2) implies (1), suppose $\Delta(\widehat{\Ecal}|_{\PBbb^2})=4 k(\Ecal)$, by Corollary \ref{computation}, we know $\widehat{\Ecal}|_{\PBbb^2}$ must be semistable of degree zero and
    $$
    p_*(\widehat{\Ecal}|_{\widehat{B_0}})\cong\Ocal_{B_0} \oplus \Ocal_{B_0}
    $$
    and 
    $$
    R^1 p_*(\widehat{\Ecal}|_{\widehat{B_0}})=0.
    $$
  Now we show that this implies
    $$
    \widehat{\Ecal}|_{\PBbb^1}\cong\Ocal_{\PBbb^1}^{\oplus 2}.
    $$
    For this, suppose 
    $$\widehat{\Ecal}|_{\PBbb^1}=\widehat{\Ecal}|_{\PBbb^1}/\tau \oplus \tau$$ 
    where $\tau$ denotes the torsion part of $\widehat{\Ecal}|_{\PBbb^1}$. We need to show that $\tau=0$ and 
    $$
    \widehat{\Ecal}|_{\PBbb^1}/\tau \cong \Ocal_{\PBbb^1}^{\oplus 2}.
    $$ 
    Consider 
    $$
    0 \rightarrow \widehat{\Ecal}|_{B_0} \rightarrow \widehat{\Ecal}|_{B_0}(1) \rightarrow (\widehat{\Ecal}|_{\PBbb^1}/\tau)(-1) \oplus \tau\rightarrow 0 
    $$
    and push it down to $B_0$ to get 
    $$
    0 \rightarrow \Ecal|^{**}_{B_0} \xrightarrow{\cong} \Ecal|^{**}_{B_0} \rightarrow H^0(\PBbb^1, (\widehat{\Ecal}|_{\PBbb^1}/\tau)(-1) \oplus \tau) \rightarrow R^1 p_*(\widehat{\Ecal}|_{\widehat{B_0}})=0.
    $$
    In particular, we have 
    $$
    H^0(\PBbb^1, (\widehat{\Ecal}|_{\PBbb^1}/\tau)(-1) \oplus \tau)=0
    $$
    which forces $\tau=0$. Thus we can assume 
    $$
    \widehat{\Ecal}|_{\PBbb^1}\cong \Ocal_{\PBbb^1}(k) \oplus \Ocal_{\PBbb^1}(-k)
    $$ 
    for some $k\geq 0$. From the above, we know
    $$
    H^0(\PBbb^1, \Ocal_{\PBbb^1}(k-1))=0
    $$
    which implies $k=0$. In particular, we must have
    $$
     \widehat{\Ecal}|_{\widehat{\PBbb^1}}\cong\Ocal_{\widehat{\PBbb^1}}^{\oplus 2}.
    $$
    This finishes the proof.
\end{proof}
This gives the new characterization of the fertile family using the discriminant. 
\begin{cor}
Given an extension $\widehat{\Ecal}$ of $\Ecal$, $\widehat{\Ecal}|_{\PBbb^2}$ is a bubble if and only if $\Delta(\widehat{\Ecal}|_{\PBbb^2})=k(\Ecal)$ and $\widehat{\Ecal}|_{\PBbb^2}^{**}\ncong \Ocal_{\PBbb^2}^{\oplus 2}$.
\end{cor}

This finishes the proof of (3) of Theorem \ref{Thm1.1}. Now we look at the relation between the semistable extension and the generic multiplicity $k^g(\Ecal)$ as mentioned in the beginning of this section.
\begin{prop}\label{Cor2.11}
$k^g(\Ecal)$ is well-defined.  The normalized semistable extension $\widehat{\Ecal}$ of $\Ecal$ satisfies
$$
\Delta(\widehat{\Ecal}|_{\PBbb^2})=k^g(\Ecal) \leq k(\Ecal).
$$
In particular, for any extension $\widehat{\Ecal}$, it satisfies
$$
\Delta(\widehat{\Ecal}|_{\PBbb^2})\leq k(\Ecal)
$$
where if the equality holds, then 
$$
\widehat{\Ecal}|_{\PBbb^1}\cong\Ocal_{\PBbb^1}^{\oplus 2}.
$$
\end{prop}

\begin{proof}
By a theorem of Grauert–M\"ulich (\cite[Page 104]{OSS11}), for a generic line $H$ in $\PBbb^2$, either one of the following holds
$$
\widehat{\Ecal}|_{H}=\Ocal_H \oplus \Ocal_H
$$  
when $c_1(\widehat{\Ecal}|_{\PBbb^2})=0$; 
$$
\widehat{\Ecal}|_{H}=\Ocal_H \oplus \Ocal_H(-1)
$$  
when $c_1(\widehat{\Ecal}|_{\PBbb^2})=-1$. Denote 
$$
B^H_0=\{z\in B: [z]\in H\}.
$$ 
Then by Proposition \ref{Thm1.5},  
$$
c_2(\widehat{\Ecal}|_{\PBbb^2})=\text{length}((\Ecal|_{B^H_0})^{**}/\Ecal|_{B^H_0}).
$$
It suffices to show that the function 
$$
H \mapsto l(H)=\dim_{\CBbb}((\Ecal|_{B^H_0})^{**}/\Ecal|_{B^H_0})
$$ 
is an upper semi-continuous function of $H \in (\PBbb^2)^*$ and generically a constant equal to $c_2(\widehat{\Ecal}|_{\PBbb^2})$. Indeed, this implies that for generic $H\in (\PBbb^2)^*$
$$
c_2(\widehat{\Ecal}|_{\PBbb^2})=l(H) \leq k(\Ecal).
$$
To prove the semicontinuity, consider $\tilde{\Ecal}$ over the incidence variety 
$$Z=\{(x, H)\in B \times (\PBbb^2)^*: x\in B^H_0 \}$$ 
and let $\pi_1: Z \rightarrow B, \pi_2: Z \rightarrow (\PBbb^2)^*$ denote the two natural projections. Consider $$
\tilde{\Ecal}=\pi_1^* \Ecal.
$$
Then $\tilde{\Ecal}$ is torsion free since $\tilde{\Ecal}|_{\pi_2^{-1}(H)}$ is canonically isomorphic to $\Ecal|_{B_0^H}$ for any $H$ which is torsion free. Consider 
$$
0\rightarrow \tilde{\Ecal} \rightarrow \tilde{\Ecal}^{**} \rightarrow \tau \rightarrow 0
$$
where 
$$\text{Supp}(\tau)\subset\{0\} \times (\PBbb^2)^*$$ 
while 
$$\text{Sing}(\tilde{\Ecal}^{**}) \subset \{0\}\times (\PBbb^2)^*$$ 
has codimension at least one since $\tilde{\Ecal}^{**}$ is reflexive and its singular set has codimension at least three. In particular, for $H\notin \pi_2(\text{Sing}(\tilde{\Ecal}^{**}))$, $\tilde{\Ecal}^{**}|_{B_0^H}$ is locally free and 
$$
l(H)=\dim_{\CBbb}\tau|_{B_0^H}.
$$
In particular, $\dim (\tau|_{B_0^H})$ is a generic constant of $H$. To show the upper semi-continuity, fix any $H_0\in (\PBbb^2)^{*}$, and pick a smooth curve $C$ in $(\PBbb^2)^*$ passing through $H_0$. We can repeat the construction above to get a family of torsion free sheaves $\tilde{\Ecal}'$ over $Z_C=\pi_2^{-1}(C) \subset Z$. Then for any $H\in C$, $(\Ecal_C)^{**}|_{\pi_2^{-1}(H)}$ is torsion free. Thus by repeating the argument above, denote $\tau'=(\Ecal_C)^{**}/\Ecal_C$
$$
\text{length}(\tau'|_{B_0^H})=\text{length}(\Ecal'^{**}|_{B_0^H} / \Ecal'|_{B_0^H}) \leq \text{length}(\Ecal'|_{B_0^H} ^{**} / \Ecal'|_{B_0^H})
$$
where the second equality follows from that $\Ecal'^{**}|_H$ is torsion free. The conclusion follows from that $\text{length}(\tau'|_{B_0^H})$ is an upper semicontinuous function of $H$.
\end{proof}

\begin{rmk}\label{G4}
For general ranks, we know 
$$
c_2(\underline{\widehat{\Ecal}}|_{\PBbb^2}) \leq k^g(\Ecal)
$$
which directly follows from Corollary \ref{computation} that can be easily adapted to general ranks.
\end{rmk}

This finishes the proof of (2) of Theorem \ref{Thm1.1}. 

\section{Formation of singularity of the type \texorpdfstring{$\Ocal \oplus \Ical$}{Ocal + Ical}}
In this section, we study the formation of singularities of the type $\Ocal_{B_0} \oplus \Ical$ where $\Ical$ is an ideal defining $0\in B_0$ with multiplicities. We will prove Theorem \ref{Main}.

\subsection{Proof for Part (1) and (2) Theorem \ref{Main}-Existence of cone families}
We fix a family $\Ecal$ satisfying 
$$
\Ecal|_{B_0}\cong \Ical \oplus \Ocal_{B_0}.
$$
Then we define 
$$
0\rightarrow \Ecal_1 \rightarrow \Ecal\rightarrow \iota_* \Ical \rightarrow 0.
$$
We first note the modification does not change the singularity formation.
\begin{lem}
    $\Ecal_1|_{B_0} \cong \Ical \oplus \Ocal_{B_0}$.
\end{lem}

\begin{proof}
We have the following exact sequence
    $$
    0\rightarrow \Ical \rightarrow \Ecal_1|_{B_0} \rightarrow \Ocal_{B_0} \rightarrow 0.
    $$
Since $\Ocal_{B_0}$ is locally free, the sequence above near $0$ is classified by 
$$
\Ext^1(\Ocal_{B_0}, \Ical)=H^1(B_0, \Ical)=0.
$$ 
Thus it splits, i.e.,  
$$
\Ecal_1|_{B_0}\cong \Ocal_{B_0} \oplus \Ical.
$$
\end{proof}
Given this, by induction, we can define $\Ecal_k$ inductively as
$$
0\rightarrow \Ecal_k \rightarrow \Ecal_{k-1} \rightarrow \iota_* \Ical \rightarrow 0.
$$
which satisfies 
$$
\Ecal_k|_{B_0}\cong \Ocal_{B_0} \oplus \Ical.
$$
Below we also denote $\widehat{\Ecal_k}$ as the normalized semistable extension of $\Ecal_k$, i.e., $\widehat{\Ecal_k}|_{\PBbb^2}$ is semistable with 
$$c_1(\widehat{\Ecal_k}|_{\PBbb^2})\in \{0,-1\}$$ 
and 
$$
p_*\widehat{\Ecal_k}\cong \Ecal_k.
$$ 
Fix $s_k$ as any extension over $B$ of the section of $\Ecal|_{B_0}$ corresponding to the $\Ocal_{B_0}$ factor. 
\begin{lem}
    $\Ecal_k$ lies in an extension as 
    $$
    0\rightarrow \Ocal_{B} \xrightarrow{s_k} \Ecal_k \rightarrow \Qcal_k \rightarrow 0
    $$
    where $\Qcal_k$ is a torsion free sheaf with $\Qcal_k|_{B_0}\cong\Ical$. In particular, for $z\neq 0$ and $|z|$ small,  $\Qcal_k|_{B_z}$ has isolated singularities and 
    $$
    l(\Qcal_k|_{B_z})=k(\Ecal).
    $$
\end{lem}

\begin{proof}
    By definition, we know $\Qcal_k|_{B_0}\cong\Ical$. Since $\Ical$ is torsion free with the non-locally free locus supported at the origin, the non-locally free locus of $\Qcal_k|_{B_z}$ must be isolated. By definition, we know 
    $$
    0\rightarrow \Ocal_{B_z} \rightarrow \Ecal_k|_{B_z} \rightarrow \Qcal_k|_{B_z} \rightarrow 0
    $$ which implies that $\Qcal_k|_{B_z}$ must be torsion free since it is part of the Koszul complex associated to the section $s_k|_{B_z}$. The conclusion about the length is a well-known fact since this implies the quotient is flat over the parameter space for $z$. We include a short proof for completeness. Indeed, write 
    $$
    0\rightarrow Q_k \rightarrow \Qcal_k^{**} \rightarrow \tau \rightarrow 0
    $$
    and we need to show that $\tau$ is flat over the parameter space for $z$. Restricting to any $z=z_0$ slice, we have the following exact sequence
    $$0\rightarrow \Tor_1(\Ocal_{B_{z_0}}, \tau)\rightarrow \Qcal_k|_{B_{z_0}}$$
    which implies 
    $$
    \Tor_1(\Ocal_{B_{z_0}},\tau)=0.
    $$
    By definition, this means $\tau$ has no torsion killed by $z-z_0$ for any $z_0$. In particular, $\tau$ is torsion free thus it is flat over the parameter space for $z$.
\end{proof}

Since $p_*\widehat{\Ecal_k}\cong \Ecal_k$, $s_k$ gives rise to a section $\widehat{s_k}$ of $\widehat{\Ecal_k}$ through the following tautological map $$
p^*p_*\widehat{\Ecal_k}=p^*\Ecal_k \rightarrow \widehat{\Ecal_k}.
$$ 
Then $\widehat{s_k}|_{\widehat{B_0}}$ gives a rank one subsheaf of $\widehat{\Ecal_k}|_{\widehat{B_0}}$. Denote $m_k$ as the vanishing order of $\widehat{s_k}|_{B_0}$ along $\PBbb^1$. We have the following natural short exact sequence
$$
0\rightarrow \Ocal_{\widehat{B_0}}(m_k)  \rightarrow \widehat{\Ecal_k}|_{\widehat{B_0}} \rightarrow \underline{\Qcal}_k \rightarrow 0.
$$

\begin{lem}\label{Lemma3.3}
$\underline{\Qcal}_k$ is torsion free over $\widehat{B_0}$ a 
and 
$$
c_1(\underline{\Qcal}_k|_{\PBbb^1}) \geq c_1(\widehat{\Ecal_k}|_{\PBbb^2})+m_k.
$$
where if $c_1(\widehat{\Ecal_k}|_{\PBbb^2})=-1$, then $m_k>0$.
\end{lem}

\begin{proof}
By definition, the following induced map 
$$\Ocal_{\widehat{B_0}}(m_k)  \rightarrow \widehat{\Ecal_k}|_{\widehat{B_0}}^{**}$$
vanishes at isolated points on $\PBbb^1$, thus its quotient must be torsion free. By definition, $\underline{\Qcal}_k$ is then a subsheaf of $\widehat{\Ecal_k}|_{\widehat{B_0}}^{**}/\Ocal_{\widehat{B_0}}(m_k)$ which is torsion free. In particular, $\underline{\Qcal}_k$ is torsion free. For the statement about the slope, the defining exact sequence above restricts to 
$$
0\rightarrow \Ocal_{\PBbb^1}(-m_k)  \rightarrow \widehat{\Ecal_k}|_{\PBbb^1} \rightarrow \underline{\Qcal}_k|_{\PBbb^1} \rightarrow 0
$$
where the exactness of the first map follows from that the kernel of the map $\Ocal_{\PBbb^1}(-m_k)  \rightarrow \widehat{\Ecal_k}|_{\PBbb^1}$ is supported at points, thus has to be zero. Thus, 
$$
c_1(\underline{\Qcal}_k|_{\PBbb^1}) = c_1(\widehat{\Ecal_k}|_{\PBbb^1})+m_k.
$$
which combined with 
$$
c_1(\widehat{\Ecal_k}|_{\PBbb^1})\geq c_1(\widehat{\Ecal_k}|_{\PBbb^2}),
$$
implies 
$$
c_1(\underline{\Qcal}_k|_{\PBbb^1}) \geq c_1(\widehat{\Ecal_k}|_{\PBbb^2})+m_k.
$$
When $c_1(\widehat{\Ecal_k}|_{\PBbb^2})=-1$, $\widehat{\Ecal_k}|_{\PBbb^2}$ is stable, thus the nontrivial subsheaf $\Ocal_{\PBbb^2}(-m_k)$ of  $\widehat{\Ecal_k}|_{\PBbb^2}$ must satisfy $m_k\geq 1$. The conclusion follows.
\end{proof}

Then we define 
$$
0 \rightarrow \widehat{\Ecal_k}' \rightarrow \widehat{\Ecal_k} \rightarrow  \iota_* (\underline{\Qcal}_k) \rightarrow 0.
$$

\begin{cor}\label{Lem3.4}
   $\widehat{\Ecal_k}'$ is reflexive and $\widehat{s_k}$ is still a section of $\widehat{\Ecal_k}'$. In particular, $p_* \widehat{\Ecal_k}'\cong \Ecal_{k+1}$.
\end{cor}

\begin{proof}
    Since $\underline{\Qcal}_k$ is torsion free, the elementary modification must be reflexive and it follows from definition that $\widehat{s_k}$ is still a section of $\widehat{\Ecal_k}'$.  Also by definition, $p_*(\widehat{\Ecal_k}')$ is isomorphic to $\Ecal_{k+1}$ away from $0$. Since both sheaves are reflexive, they must be isomorphic. 
\end{proof}

\begin{lem}\label{Lemma3.5}
    $\widehat{\Ecal_k}'|_{\PBbb^2}$ lies in the following exact sequence
    $$
0 \rightarrow \widehat{\Ecal_k}'|_{\PBbb^2} \rightarrow \widehat{\Ecal_k}|_{\PBbb^2} \rightarrow  \iota_* (\underline{\Qcal}_k|_{\PBbb^1} ) \rightarrow 0.
$$
In particular, 
$$
\Delta(\widehat{\Ecal_k}'|_{\PBbb^2})\geq \Delta_{\Ecal_{k+1}}-1.$$
\end{lem}

\begin{proof}
 Since pulling-back is right exact, by restricting 
    $$
0 \rightarrow \widehat{\Ecal_k}' \rightarrow \widehat{\Ecal_k} \rightarrow  \iota_* (\underline{\Qcal}_k|_{\widehat{B_0}}) \rightarrow 0
$$ 
to $\PBbb^2$, we get the short exact sequence 
  $$
\widehat{\Ecal_k}'|_{\PBbb^2} \rightarrow \widehat{\Ecal_k}|_{\PBbb^2} \rightarrow  (\iota_* (\underline{\Qcal}_k|_{\widehat{B_0}}))|_{\PBbb^2} \rightarrow 0.
$$ 
Then the kernel of the first map is supported at points. Since $\widehat{\Ecal_k}'|_{\PBbb^2}$ is torsion free, it has to be zero. Also, it follows from the definition that 
$$
(\iota_* (\underline{\Qcal}_k|_{\widehat{B_0}}))|_{\PBbb^2}=\iota_* (\underline{\Qcal}_k|_{\widehat{B_0}\cap \PBbb^1})=\iota_*(\underline{\Qcal}_k|_{\PBbb^1}).
$$
In particular, we have the exact sequence claimed. Now the statement about the relation between the discriminant follows from a direct computation by using the obtained short exact sequence and Lemma \ref{Lemma3.3}. Indeed, using the fact that 
$$
c(\iota_* (\underline{\Qcal}_k|_{\PBbb^1}))=1+a+c_1(\underline{\Qcal}_k|_{\PBbb^1}) a^2
$$
where $a=c_1(\Ocal_{\PBbb^1}(1))$ and by the multiplicative property of the total Chern classes, we have the following  
$$
\Delta(\widehat{\Ecal_k}'|_{\PBbb^2})=\Delta(\widehat{\Ecal_k}|_{\PBbb^2})-1
+2(2c_1(\underline{\Qcal}_k|_{\PBbb^1})-c_1(\widehat{\Ecal_k}|_{\PBbb^2})).$$
By Lemma \ref{Lemma3.3}, we know 
$$
\Delta(\widehat{\Ecal_k}'|_{\PBbb^2})\geq \Delta(\widehat{\Ecal_k}|_{\PBbb^2})-1+4m_k+2c_1.
$$
When $c_1(\widehat{\Ecal_k}|_{\PBbb^2})=0$, it is trivial that $\Delta(\widehat{\Ecal_k}'|_{\PBbb^2})\geq \Delta(\widehat{\Ecal_k}|_{\PBbb^2})-1$. When $c_1(\widehat{\Ecal_k}|_{\PBbb^2})=-1$, we know from Lemma \ref{Lemma3.3} that $m_k>0$ which implies $\Delta(\widehat{\Ecal_k}'|_{\PBbb^2})>\Delta(\widehat{\Ecal_k}|_{\PBbb^2})$. The conclusion follows.
\end{proof}

\begin{cor}\label{Cor3.6}
For any $k$, 
$$
\Delta_{\Ecal_{k+1}} \geq \Delta_{\Ecal_k}.
$$ 
In particular, for $k$ large, 
\begin{itemize}
    \item $m_k=0$;
    \item $c_1(\widehat{\Ecal_{k}}|_{\PBbb^2})=0;$
    \item $\Delta_{\Ecal_{k+1}}=\Delta_{\Ecal_k}.$
\end{itemize}
\end{cor}

\begin{proof}
    By definition, $\widehat{s_k}|_{\PBbb^2}$ is a section of $\widehat{\Ecal_k}'|_{\PBbb^2}$ where 
    $$
    c_1(\widehat{\Ecal_k}'|_{\PBbb^2})<0.
    $$
    Thus $\widehat{\Ecal_k}'|_{\PBbb^2}$ is unstable. Now as already noted in the proof of Proposition \ref{Prop2.6}, the elementary modification along its maximal destabilizing rank one subsheaf increases the discriminant by at least $1$. In particular, by Lemma \ref{Lemma3.5}, we have 
    $$
    \begin{aligned}
    \Delta_{\Ecal_{k+1}}&\geq \Delta(\widehat{\Ecal_k}'|_{\PBbb^2})+1\\ 
    &\geq \Delta_{\Ecal_{k}}.
    \end{aligned}
    $$
    Since $\Delta_{\Ecal_l} \leq k(\Ecal)$ for any $l$ by Proposition \ref{Cor2.11} and $\Delta_{\Ecal_{l}}$ is a sequence of non-decreasing integers, we must have 
    $$\Delta_{\Ecal_{k+1}}=\Delta_{\Ecal_k}$$ 
    for $k$ large. This forces 
    $$m_k=0$$
    and
    $$c_1(\widehat{\Ecal_k}|_{\PBbb^2})=0$$
    for $k$ large. The conclusion follows.
\end{proof}

Denote $k_0$ as the first integer so that 
$$\Delta_{\Ecal_k}=\Delta_{\Ecal_{k_0}}$$ 
for any $k\geq k_0$. 
\begin{lem}\label{Lemma3.7}
    For $k\geq k_0$, 
$$
0\rightarrow \Ocal_{\widehat{B}} \xrightarrow{\widehat{s_k}} \widehat{\Ecal_{k}} \rightarrow \widehat{\Qcal_{k}} \rightarrow 0
$$
where $\widehat{\Qcal_k}|_{\PBbb^2}$ is torsion free and $\widehat{\Qcal_k}$ is isomorphic to $p^* \Qcal$ away from the exceptional divisor $\PBbb^2$.
\end{lem}

\begin{proof}
    Since $\widehat{\Ecal_k}|_{\PBbb^2}$ is semistable of degree $0$ by Corollary \ref{Cor3.6}, $\widehat{s_k}|_{\PBbb^2}$ vanishes at isolated points when viewed as a section of $\widehat{\Ecal_k}|_{\PBbb^2}^{**}$. In particular, the induced quotient $\widehat{\Ecal_k}|_{\PBbb^2}^{**}/\Ocal_{\PBbb^2}$ is torsion free, thus $\widehat{\Ecal_k}|_{\PBbb^2}/\Ocal_{\PBbb^2}\cong \widehat{\Qcal_k}|_{\PBbb^2}$ is torsion free since it is a subsheaf of $\widehat{\Ecal_k}|_{\PBbb^2}^{**}/\Ocal_{\PBbb^2}$. In particular, $\widehat{\Qcal_k}$ is torsion free.
\end{proof}

\begin{lem}\label{prop3.9}
    $c_2(\widehat{\Qcal_k}|_{\PBbb^2})\geq \lim\sup_{z\rightarrow 0} l(p_*(\Qcal_k)|_{B_z})=k(\Ecal)$. 
\end{lem}

\begin{proof}
Denote 
$$
0\rightarrow \widehat{\Qcal_k} \rightarrow (\widehat{\Qcal_k})^{**}\cong \Ocal_{\widehat{B}} \rightarrow \tau_k \rightarrow 0.
$$
Then by assumption, since $p_*(\widehat{\Qcal_k})|_{B_0}$ is locally free away from $0\in B_0$, we know 
$$
Supp(\tau_k) \subset (\widehat{B} \setminus \widehat{B_0}) \cup \PBbb^1.
$$
Let $\widehat{B_z}$ be the strict transform of $B_z$ in $\widehat{B}$. Then by definition
$$
l(p_*(\Qcal_k)|_{B_z})=l(\tau_k|_{\widehat{B_z}}),
$$
while 
$$
c_2(\widehat{\Qcal_k}|_{\PBbb^2})=l(\tau_k|_{\PBbb^2}).
$$
Since the support of $\tau_k|_{\widehat{\CBbb^2_z}}$ converges to a subset of $\PBbb^2$ as $z \rightarrow 0$, by semicontinuity, we know 
$$
l(\tau_k|_{\PBbb^2})\geq \limsup_{z\rightarrow 0} l(\tau_k|_{\widehat{B_z}}).
$$
The conclusion follows.
\end{proof}

\begin{cor}\label{Cor3.9}
For $k\geq k_0$, $\Ecal_k$ are nonbarren; for any $k>k_0$, $\Ecal_{k}$ are all cone families with isomorphic cones. 
\end{cor}

\begin{proof}
By Lemma \ref{prop3.9}, for any $k\geq k_0$, we have 
$$
\Delta_{\Ecal_{k}}\geq  \lim\sup_{z\rightarrow 0} l(p_*(\Qcal_k)|_{B_z})=k(\Ecal).
$$
Thus the equality must hold, i.e., 
$$
\Delta_{\Ecal_{k}}=k(\Ecal).
$$ 
which by Proposition \ref{BubbleCharaterization} implies
$$
\widehat{\Ecal_k}|_{\PBbb^1}\cong \Ocal_{\PBbb^1}^{\oplus 2},
$$
i.e., $\Ecal_{k}$ must be non-barren. By definition,  $\widehat{\Ecal_{k_0}}'|_{\PBbb^2}$ lies in the following exact sequence 
$$
0\rightarrow \Ocal_{\PBbb^2} \rightarrow \widehat{\Ecal_{k_0}}'|_{\PBbb^2} \rightarrow \widehat{\Qcal_{k_0}}|_{\PBbb^2}(-1) \rightarrow 0.
$$
Then $\widehat{\Ecal_{k+1}}$ is the elementary modification of $\widehat{\Ecal_{k_0}}'$ along $\Ocal_{\PBbb^2}$. In particular, it satisfies 
$$
0\rightarrow \widehat{\Qcal_{k_0}}|_{\PBbb^2} \rightarrow \widehat{\Ecal_{k+1}}|_{\PBbb^2} \rightarrow \Ocal_{\PBbb^2} \rightarrow 0.
$$
Since by definition $\widehat{s_k}$ as a section of $\widehat{\Ecal_{k+1}}$ restricts to a section of $\widehat{\Ecal_k}$ which maps onto $\Ocal_{\PBbb^2}$ through the short exact sequence above, it must split, i.e., 
$$
\widehat{\Ecal_{k+1}}|_{\PBbb^2} \cong \Ocal_{\PBbb^2}\oplus \widehat{\Qcal_{k_0}}|_{\PBbb^2}.
$$
Now the statement for general $k$ follows from induction. 
\end{proof}

This finishes the proof of Part $(1)$ and $(2)$ of Theorem \ref{Main}.

\begin{rmk}
    In general, it could happen that $\Ecal_{k_0}$ is a cone family. For example, we can start with $\Ecal$ being a cone family, then $k_0=0$ in this case. In order to get the fertile family, we have to use different modifications, which is the goal of the next section.
\end{rmk}

\subsection{Proof for Part (3) Theorem \ref{Main}-Existence of fertile families}
Below we suppose $\Ecal$ is a \emph{cone} family with 
$$
\Ecal|_{B_0}\cong\Ical \oplus \Ocal_{B_0}.
$$
Then we define $\Ecal_{-k}$ inductively as follows 
$$
0\rightarrow \Ecal_{-k-1} \rightarrow \Ecal_{-k} \rightarrow \iota_* \Ocal_{B_0} \rightarrow 0
$$
under the inductive \emph{assumption} that 
$$
\Ecal_{-k}|_{B_0}\cong\Ocal_{B_0} \oplus \Ical.
$$ 
Thus $\Ecal_{-1}$ is always well-defined. 

\begin{rmk}
The key conclusion below is that $\Ecal_{-1}$ forms the same singularity as $\Ecal$ so that by induction, we can always keep modifying the family without changing the singularity if it is a cone family.
\end{rmk}

By Corollary \ref{Cor3.9}, we know the semistable extension $\widehat{\Ecal_{-k}}$ satisfies 
$$
\widehat{\Ecal_{-k}}|_{\PBbb^2}\cong\Ocal_{\PBbb^2} \oplus \underline{\Ical}
$$
for some fixed sheaf $\underline{\Ical}$. By Lemma \ref{Lemma3.7}, we can also write 
$$
0\rightarrow \Ocal_{\widehat{B}} \rightarrow \widehat{\Ecal_{-k}} \rightarrow \widehat{\Qcal_{-k}} \rightarrow 0
$$
which if pushed down to $B$ gives 
$$
0\rightarrow \Ocal_{B} \rightarrow \Ecal_{-k} \rightarrow \Qcal_{-k} \rightarrow 0
$$
since $R^1p_* \Ocal_{\widehat{B}}=0$. We also know the first short exact sequence above over $\widehat{B}$ restricts to a splitting exact sequence over $\widehat{B_0}$ and we fix a splitting map for it as 
$$
\widehat{\Ecal_{-k}}|_{\widehat{B_0}} \rightarrow \Ocal_{\widehat{B_0}},
$$
which restricts to the natural pull-back over $\widehat{B_0}$ of the projection map 
$$\Ecal|_{\widehat{B_0}} \rightarrow \Ocal_{B_0}$$ 
given by the splitting we fixed in the above.

Consider
$$
0 \rightarrow \widehat{\Ecal_{-k}}' \rightarrow \widehat{\Ecal_{-k}} \rightarrow \iota_* \Ocal_{\widehat{B_0}}\rightarrow 0.
$$
which if pushed down to $B$ gives 
$$
0\rightarrow p_* \widehat{\Ecal_{-k}}' \rightarrow \Ecal_{-k} \rightarrow \iota_* \Ocal_{B_0} \rightarrow 0
$$

\begin{lem}
$p_* \widehat{\Ecal_{-k}}'\cong \Ecal_{-k-1}$.
\end{lem}

\begin{proof}
    This follows from the same argument as Lemma \ref{Lem3.4}.
\end{proof}

\begin{lem}\label{Lemma3.13}
$\widehat{\Ecal_{-k-1}}'|_{\PBbb_2}=\underline{\Qcal}\oplus \Ocal_{\PBbb^2}(-1)$. In particular, the semistable extension $\widehat{\Ecal_{-k-1}}$ of $\Ecal_{-k-1}$ satisfies 
$$
0\rightarrow \Ocal_{\PBbb^2} \rightarrow \widehat{\Ecal_{-k-1}}|_{\PBbb^2} \rightarrow \underline{\Qcal} \rightarrow 0.
$$
\end{lem}

\begin{proof}
By definition, we know 
$$
0 \rightarrow \widehat{\Ecal}'_{-k-1} \rightarrow \widehat{\Ecal}_{-k} \rightarrow \iota_* \Ocal_{\widehat{B_0}}\rightarrow 0.
$$
which restricts to $\PBbb^2$ as 
$$
0 \rightarrow \widehat{\Ecal}'_{-k-1}|_{\PBbb^2} \rightarrow \widehat{\Ecal}_{-k}|_{\PBbb^2} \rightarrow \iota_* \Ocal_{\PBbb^1}\rightarrow 0.
$$
By definition, we know 
$$
\widehat{\Ecal}'_{-k-1}|_{\PBbb_2}\cong \underline{\Qcal}\oplus \Ocal_{\PBbb^2}(-1).
$$
Then the elementary modification of $\widehat{\Ecal_{-k-1}}'$ along $\underline{\Qcal}$ gives the optimal extension $\widehat{\Ecal_{-k-1}}$.
\end{proof}

This gives the following \emph{key} property.

\begin{cor}\label{Cor3.9}
Suppose $\Ecal_{-k}$ is a cone family, then the singularity formation remains the same for $\Ecal_{-k-1}$, i.e.,
$$
\Ecal_{-k-1}|_{B_0}\cong \Ocal_{B_0} \oplus \Ical.
$$
\end{cor}

\begin{proof}
   By Lemma \ref{Lemma3.13}, we know 
    $$
    \Delta(\Ecal_{-k-1})=\Delta(\Ecal_{-k}).
    $$
    By Proposition \ref{BubbleCharaterization}, 
    $$
    k(\Ecal_{-k-1})=\Delta_{\Ecal_{-k-1}}=\Delta_{\Ecal_{-k}}=\dim_{\CBbb} \Ical^{**}/\Ical
    $$
    This forces the natural exact sequence 
    $$
    0\rightarrow \Ocal_{B_0} \rightarrow \Ecal_{-k-1}|_{B_0} \rightarrow \Ical \rightarrow 0
    $$
    to split, which follows from that the natural induced map 
    $$
    \Ocal_{B_0} \rightarrow \Ecal_{-k-1}|_{B_0}^{**}
    $$
    is nowhere vanishing as a bundle map. Indeed, we can fix any surjective map $\Ecal_{-k-1}|_{B_0}^{**}\rightarrow \Ocal_{B_0}$ which restricts to identity on $\Ocal_{B_0}$. Then this map restricts to $\Ecal_{-k-1}|_{B_0}\rightarrow \Ocal_{B_0}$ that splits the sequence above. Now we show the natural induced map is a bundle map indeed. For this, by assumption, the following natural induced map
    $$
    \Ecal_{-k-1}|_{B_0}^{**}/\Ecal_{-k-1}|_{B_0} \rightarrow \Ical^{**}/\Ical
    $$
    is an isomorphism by dimensional reasons. This implies the map 
    $$
    \Ecal_{-k-1}|_{B_0}^{**} \rightarrow \Ical^{**}
    $$
    is surjective, thus the map $\Ocal_{B_0} \rightarrow \Ecal_{-k-1}|_{B_0}^{**}$ is injective as a bundle map. This finishes the proof. 
\end{proof}

Suppose $\Ecal$ is a cone family, we can then inductively conclude the following 
\begin{prop}
There exists a unique $k_0$ so that $\Ecal_{-k_0}$  is a fertile family forming the singularity $\Ocal_{B_0} \oplus \Ical$. In particular, for $k < k_0$, $\Ecal_{-k}$ are all cone families.
\end{prop}

\begin{proof}
It suffices to show that $\Ecal_{-k_0}$ are not cone families for some $k_0$. We argue by contradiction. Otherwise, $\Ecal_{-k}$ are all cone families. The original cone family lies in the following exact sequence 
$$
0\rightarrow \Ocal_{B} \xrightarrow{s} \Ecal \rightarrow \Qcal \rightarrow 0
$$
which restricts to 
$$
0\rightarrow \Ocal_{B_0} \xrightarrow{s|_{B_0}} \Ecal|_{B_0} \rightarrow \Ical \rightarrow 0
$$
that splits as $\Ecal|_{B_0} \cong \Ocal_{B_0} \oplus \Ical.$ We can then represent $\Ecal$ in the following exact sequence 
$$
0\rightarrow \Rcal \rightarrow \Ocal_{B}^{\oplus n} \oplus \Ocal_{B} \xrightarrow{(\sigma, s)} \Ecal \rightarrow 0
$$
for some choice of sections $\sigma$ of $\Ecal$ that generate $\Qcal$. By definition, the elementary modification $\Ecal_{-1}$ of $\Ecal$ along the $\Ical$ lies in  the following commutative diagram 
\[\begin{tikzcd}
	0 & {\Ocal_B} & {\Ocal_B^{\oplus n}\oplus \Ocal } & \Ecal & 0 \\
	0 & {\Ocal_B} & {\Ocal_B^{\oplus n}\oplus (z)} & {\Ecal_{-1}} & 0
	\arrow[from=1-1, to=1-2]
	\arrow[from=1-2, to=1-3]
	\arrow["{(\sigma, s)}", from=1-3, to=1-4]
	\arrow[from=1-4, to=1-5]
	\arrow[from=2-1, to=2-2]
	\arrow["{=}"', from=2-2, to=1-2]
	\arrow[from=2-2, to=2-3]
	\arrow[from=2-3, to=1-3]
	\arrow["{(\sigma, s)}", from=2-3, to=2-4]
	\arrow[from=2-4, to=1-4]
	\arrow[from=2-4, to=2-5]
\end{tikzcd}\]
which gives 
$$
0\rightarrow (z) \xrightarrow{s} \Ecal_{-1} \rightarrow \Qcal \rightarrow 0.
$$
By our assumption, the restriction
$$
0\rightarrow (z)/(z^2) \rightarrow \Ecal_{-1}|_{B_0} \rightarrow \Ical \rightarrow 0
$$
splits. $\Ecal_{-2}$ is then defined by picking a projection of $\Ecal_{-1}|_{B_0}$ to the $(z)/(z^2)$ factor which restricts to identity as map from $(z)/(z^2)$ to $(z)/(z^2)$, i.e., a short exact sequence 
$$
\Ecal_{-1} \rightarrow \iota_* (z)/(z^2) \rightarrow 0
$$
so that the composition map 
$$(z)/(z^2) \rightarrow \Ecal_{-1}|_{B_0} \rightarrow (z)/(z^2)$$
is the identity map. Then by definition, we know 
$$
\Ecal_{-2} \cong \Ocal_{B}^{\oplus n}\oplus (z^2)/\Rcal.
$$
By induction, we know 
$$
\Ecal_{-k} \cong \Ocal_{B}^{\oplus n}\oplus (z^k)/\Rcal
$$
for any $k$. In particular, $\Rcal\subset \Ocal_{B}^{\oplus n}$, thus 
$$
\Ecal\cong \Ocal_{B} \oplus \Qcal
$$
which contradicts to $\Ecal$ having an essential singularity at $0$. The conclusion follows.
\end{proof}

\section{Bubbling with low multiplicities}
In this section, we study bubble trees for singularity formation with low multiplicities.
\subsection{Bubbling of multiplicity one}

The simplest singularity formed for rank two bundles over surfaces is those of multiplicity one. Even in this case, as we have noted already (see the proof of Proposition \ref{BarrenFamily}),  a family could still be barren. 

The following is straightforward.
\begin{lem}
    Suppose $\Ecal$ is a family of sheaves forming a singularity with $k(\Ecal)=1$, then 
    $$
    \Ecal|_{B_0}\cong \Ocal \oplus \Ical
    $$
    where $\Ical \cong (x,y)$.
\end{lem}

Thus, we can assume that the family forms the singularity $\Ocal_{B_0} \oplus (x,y)$. As expected, this corresponds to the basic one instanton in gauge theory (\cite{DonaldsonKronheimer:90}).

\begin{lem}\label{BasicOne}
    The family
    $$
    0\rightarrow \Ocal \xrightarrow{\begin{pmatrix}
    x\\
    y\\
    z^n
    \end{pmatrix}} \Ocal^{\oplus 3} \rightarrow \Ecal \rightarrow 0
    $$
    is fertile if and only if $n=2.$
\end{lem}

\begin{proof}
    If $n=2$, the optimal extension is given by 
      $$
    0\rightarrow \Ocal_{\widehat{B}}(2) \xrightarrow{\begin{pmatrix}
    X\\
    Y\\
    Z^2
    \end{pmatrix}} \Ocal_{\widehat{B}}(1)^{\oplus 2}\oplus \Ocal_{\widehat{B}} \rightarrow \widehat{\Ecal} \rightarrow 0
    $$
    where $\widehat{\Ecal}|_{\PBbb^2}$ is exactly the basic one instanton. If $n=1$, then the optimal extension is given by 
     $$
    0\rightarrow \Ocal_{\widehat{B}}(2) \xrightarrow{\begin{pmatrix}
    X\\
    Y\\
    Z
    \end{pmatrix}} \Ocal_{\widehat{B}}(1)^{\oplus 3} \rightarrow \widehat{\Ecal} \rightarrow 0
    $$
    where $\widehat{\Ecal}|_{\PBbb^2}=\Tcal_{\PBbb^2}(-2)$ which is stable but the family is not fertile since 
    $$
    \widehat{\Ecal}|_{\PBbb^2}\cong \Ocal_{\PBbb^1}(-1) \oplus \Ocal_{\PBbb^1}.
    $$
    For $n\geq 3,$ the following gives the cone extension
    $$
    0\rightarrow \Ocal_{\widehat{B}}(2) \xrightarrow{\begin{pmatrix}
    X\\
    Y\\
    z^{n-2}Z^2
    \end{pmatrix}} \Ocal_{\widehat{B}}(1)^{\oplus 2} \oplus \Ocal_{\widehat{B}}(2)  \rightarrow \widehat{\Ecal} \rightarrow 0
    $$
    where 
    $$
    \widehat{\Ecal}|_{\PBbb^2}\cong \Ocal_{\PBbb^2}\oplus \Ical_{[0,0,1]}.
    $$
\end{proof}

Now we look at a simple example  
\begin{exam}
    Consider the following fertile family   
    $$
    0\rightarrow \Ocal \xrightarrow{\begin{pmatrix}
    x+z\\
    y+z\\
    z^2
    \end{pmatrix}} \Ocal^{\oplus 3} \rightarrow \Ecal \rightarrow 0.
    $$
    This forms the singularity $\Ocal \oplus (x,y)$ and the bubble $\underline{\Ecal}$ is given by 
    $$
    0\rightarrow \Ocal \xrightarrow{\begin{pmatrix}
    X+Z\\
    Y+Z\\
    Z^2
    \end{pmatrix}} \Ocal(1) \oplus \Ocal(1) \oplus \Ocal(2) \rightarrow \underline{\Ecal}(2)\rightarrow 0.
    $$
    This is different from the standard bubble
    $$
    0\rightarrow \Ocal \xrightarrow{\begin{pmatrix}
    X\\
    Y\\
    Z^2
    \end{pmatrix}} \Ocal(1) \oplus \Ocal(1) \oplus \Ocal(2) \rightarrow \underline{\Ecal}(2)\rightarrow 0
    $$
    since $X+Z=0$ is a jumping line for the first but not for the standard one. On the other hand, the two bubbles can be related by doing a coordinate change of $\PBbb^2$ as $X\mapsto X+Z, Y \mapsto Y+Z, Z\mapsto Z$ induced by the coordinate change o as $(x,y,z) \mapsto (x+z,y+z,z)$.
\end{exam}

This simple example explains in general what happens for forming singularity of multiplicity $1$.

\begin{lem}\label{lem4.3}
Given any family $\Ecal$ forming the singularity $\Ocal \oplus (x,y)$, $\Ecal$ lies in the following exact sequence
    $$
    0\rightarrow \Ocal \xrightarrow{\begin{pmatrix}
    p_1\\
    p_2\\
    p_3
    \end{pmatrix}} \Ocal^{\oplus 3} \rightarrow \Ecal \rightarrow 0.
    $$
    where $p_1|_{z=0}=x$ and $p_2|_{z=0}=y$. 
\end{lem}

\begin{proof}
    By assumption, we have a short exact sequence 
    $$
    \Ocal_{B}^{\oplus 3} \xrightarrow{\phi=\begin{pmatrix}
        x&y&0
    \end{pmatrix}} \Ecal|_{z=0} \rightarrow 0.
    $$
    Since $H^1(B, \Ecal)=0$, this can be extended to be map 
    $$
    \Ocal^{\oplus 3} \xrightarrow{\Phi} \Ecal. 
    $$
    Since $\Phi|_{z=0}$ is surjective, by Nakayama's lemma, we know $\Phi$ is surjective. Since $\Ecal$ has rank two, $\ker(\Phi)$ is a rank one reflexive sheaf which has to be locally free. This gives what we need.
\end{proof}

\begin{prop}\label{ChargeOne}
There exists an essentially unique fertile family forming the singularity $\Ocal \oplus (x,y)$. More precisely, up to isomorphisms, all such families can be obtained by pulling back the family  
    $$
    0\rightarrow \Ocal \xrightarrow{\begin{pmatrix}
    x\\
    y\\
    z^2
    \end{pmatrix}} \Ocal^{\oplus 3} \rightarrow \Ecal \rightarrow 0.
    $$
through a local coordinate change near $0$ in $B$.
\end{prop}

\begin{proof}
    Given any family $\Ecal$ forming the singularity $\Ocal \oplus (x,y)$, by Lemma \ref{lem4.3}, we can represent $\Ecal$ as 
    $$
    0\rightarrow \Ocal \xrightarrow{\begin{pmatrix}
    p_1\\
    p_2\\
    p_3
    \end{pmatrix}} \Ocal^{\oplus 3} \rightarrow \Ecal \rightarrow 0.
    $$
    where $p_1|_{z=0}=x$ and $p_2|_{z=0}=y$. In particular, the map 
    $$(x,y,z)\mapsto (p_1, p_2,z)$$ 
    induces a local coordinate change near $0$. After the coordinate change, the new family is given by 
    $$
    0\rightarrow \Ocal \xrightarrow{\begin{pmatrix}
    x\\
    y\\
    p_3'
    \end{pmatrix}} \Ocal^{\oplus 3} \rightarrow \Ecal \rightarrow 0.
    $$
    By applying a linear transform to $\Ocal^{\oplus 3}$, we know the family above is isomorphic to the family
    $$
    0\rightarrow \Ocal \xrightarrow{\begin{pmatrix}
    x\\
    y\\
    z^n
    \end{pmatrix}} \Ocal^{\oplus 3} \rightarrow \Ecal \rightarrow 0.
    $$
    for some $n$, which is fertile if and only if $n=2$ by Lemma \ref{BasicOne}. The conclusion follows.
\end{proof}

\subsection{Bubbling of multiplicity two}
In this section, we will focus on families forming the singularity $(x, y^2) \oplus \Ocal_{B}$.  

It is fairly easy to construct fertile families forming $\Ocal \oplus (x,y^2)$ with the height of the bubble tree being one as we will see below. Instead, we first ask the following more interesting question

\begin{prob}
    Does there exist a fertile family forming the singularity 
    $$\Ocal \oplus (x,y^2)$$ 
    with the height of the bubble tree being at least two?
\end{prob}

The answer is yes. For this, consider
    $$
0\rightarrow \Ocal \xrightarrow{\begin{pmatrix}
    x\\
    y^2+z^4\\
    yz^2
\end{pmatrix}} \Ocal^{\oplus 3} \rightarrow \Ecal \rightarrow 0.
$$
The bubble of this family is given by 
    $$
0\rightarrow \Ocal \xrightarrow{\begin{pmatrix}
    X\\
    Y^2\\
    YZ^2
\end{pmatrix}} \Ocal(1) \oplus \Ocal(2) \oplus \Ocal(3) \rightarrow \underline{\Ecal}(3) \rightarrow 0.
$$
Note $\underline{\Ecal}$ has a singularity near $[0,0,1]$. Thus the height of the bubble tree is two where the bubble of each level has charge one, i.e., they are both modeled on the standard bubble of charge one. Note here the bubble near the singular point $[0,0,1]$ is modeled on 
$$
0\rightarrow \Ocal \xrightarrow{\begin{pmatrix}
    x\\
    y
\end{pmatrix}}  \Ocal \oplus (x,y)  \xrightarrow{\begin{pmatrix}
    -y&x
\end{pmatrix}} (y,x^2) \rightarrow 0.
$$

\begin{prob}
    Can we classify all the fertile families forming singularity $\Ocal \oplus (x,y^2)$?
\end{prob}

The complete answer to this question does not seem straightforward, but we note the example above can be generalized as
$$
0\rightarrow \Ocal \xrightarrow{\begin{pmatrix}
    x\\
    y^2+z^m\\
    yz^2
\end{pmatrix}} \Ocal^{\oplus 3} \rightarrow \Ecal \rightarrow 0.
$$
where $m\geq 3$. The bubble trees are the same while the families are totally different.  Also we know from this family that

\begin{cor}
In general, the bubble tree is not an invariant of $\Ecal|_{mB_0}$ for any positive integer $m$ depending only on the singularity and the bubble tree.
\end{cor}

Given this, it is more realistic for us to classify the bubble trees. We first note 

\begin{prop}
    The bubble tree of height two for forming $(x^2,y)\oplus \Ocal$ has height two with the double dual of the bubble at each level being the standard bubble of charge one.
\end{prop}

\begin{proof}
    This follows from Proposition \ref{ChargeOne}, i.e, the uniqueness of bubbles of charge one.
\end{proof}

Next we focus on fertile families with the height of bubble tree being one. We first note the following
\begin{prop}
    Suppose $\Ecal$ is a fertile family forming $(x^2,y)\oplus \Ocal_{B_0}$ with the height of the bubble tree being one. Then the optimal extension $\underline{\widehat{\Ecal}}$ lies in the following exact sequence  
    $$
    0\rightarrow \Ocal_{\PBbb^2} \rightarrow \underline{\widehat{\Ecal}} \rightarrow \underline{\Ical} \rightarrow 0
    $$
    where $\underline{\Ical}$ is isomorphic to the homogenization of $\Ical$. In particular, they are given by the constructions above.
\end{prop}

\begin{proof}
    By Theorem \ref{Main}, we know the optimal extension $\underline{\widehat{\Ecal}}$ lies in the following exact sequence  
    $$
    0\rightarrow \Ocal_{\PBbb^2} \rightarrow \underline{\widehat{\Ecal}} \rightarrow \underline{\Ical} \rightarrow 0
    $$
    for some ideal sheaf $\underline{\Ical}$ on $\PBbb^2$ which defines a subscheme supported at $[0,0,1]$. It suffices to show that it is isomorphic to $(x^2,y)$ under suitable choice of coordinates. For this, consider $\Ocal_{\PBbb^2}/\underline{\Ical}$, by choosing local coordinates $(u,v)$ where $[u,v,1]\in \PBbb^2$. First, we note $su+tv\in \underline{\Ical}$ for some constants $s,t$. Otherwise, $1, \bar{u}, \bar{v} \in \Ocal_{\PBbb^2}/\underline{\Ical}$, thus it is of dimension three, which is a contradiction. WLOG, assume $u\in \underline{\Ical}$. Now $\Ocal_{\PBbb^2}/\underline{\Ical}$ must be generated by $1$ and $\bar{v}$. In particular, $v^2\in \underline{\Ical}$. The conclusion follows.
\end{proof}

Given this, we focus now on constructing smooth bubbles of charge two. For this, let $\underline{\Ical}$ be the homogenization of $(x,y^2)$. Then the space $\Ext^1(\underline{\Ical}, \Ocal_{\PBbb^2})$ will give us the bubbles. For this, take the resolution
$$
0\rightarrow \Ocal_{\PBbb^2} \xrightarrow{\begin{pmatrix}
    X\\
    Y^2
\end{pmatrix}} \Ocal_{\PBbb^2}(1) \oplus \Ocal_{\PBbb^2}(2) \rightarrow \underline{\Ical}(3)\rightarrow 0.
$$
Then any semi-stable extension is given by a polynomial of degree $k$ 
$$
\Ocal_{\PBbb^2} \xrightarrow{p_3} \Ocal_{\PBbb^2}(k). 
$$
where $k$ is to be specified. More precisely, the extension associated to this polynomial $p_3$ is given by 
$$
0\rightarrow \Ocal \xrightarrow{\begin{pmatrix}
    X\\
    Y^2\\
    p_3
\end{pmatrix}} \Ocal_{\PBbb^3}(1) \oplus \Ocal_{\PBbb^3}(2) \oplus \Ocal_{\PBbb^3}(k) \rightarrow \underline{\Ecal}(3+k) \rightarrow 0.
$$
By doing a linear transform, we can assume 
$$
p_3=a Z^k+ b YZ^{k-1}
$$
To make the family fertile, the following must also hold 
$$
\underline{\Ecal}|_{\PBbb^1} \cong \Ocal_{\PBbb^1}^{\oplus 2},
$$
which forces $k=3$, i.e.,   
$$
p_3=a Z^3+ b YZ^{2}.
$$
Note by scaling the vector $(a,b)$, the bubble stays in the same isomorphism class. In particular, this gives 
\begin{prop}
    The space of bubble trees when forming singularities $\Ocal \oplus (x,y^2)$ can be parametrized by $t\in \PBbb^1=\CBbb \cup \infty$ where 
    \begin{itemize}
        \item the bubble tree has height one when $t\neq 0$;
        \item the bubble tree has height two when $t=0$.
    \end{itemize}
\end{prop}

\section{More examples}
In this section, we study more examples and use them to give negative answers to some general plausible questions. Below for a fertile family $\Ecal$ forming the singularity $\Ocal\oplus \Ical$, we denote the first level bubble as $\underline{\Ecal}$ which by Theorem \ref{Main} lies in the following exact sequence 
$$
0\rightarrow \Ocal_{\PBbb^2} \rightarrow \underline{\Ecal} \rightarrow \underline{\Ical} \rightarrow 0
$$
for some ideal sheaf $\underline{\Ical}$ which defines a subscheme of $\PBbb^2$ supported at $[0,0,1]$.

We first give a fertile family forming the singularity $(x^2,y^2,xy) \oplus \Ocal_{B_0}$ where the ideal $\underline{\Ical}$ is a locally complete intersection thus different from the original ideal.
\begin{exam}
Consider the family $\Ecal$ given by
    $$
    0\rightarrow \Ocal^{\oplus 2} \xrightarrow{\begin{pmatrix}
        0&x\\
        x&y\\
        y&z\\
        z^3&0
    \end{pmatrix}} \Ocal^{\oplus 4} \rightarrow \Ecal \rightarrow 0.
    $$
Note 
$$\Ecal|_{z=0}=\Ocal \oplus (x^2,y^2,xy)$$ 
and $\Ecal$ is fertile with the bubble tree being height one and the first level bubble being given by 
    $$
    0\rightarrow \Ocal^{\oplus 2} \xrightarrow{\begin{pmatrix}
        0&X\\
        X&Y\\
        Y&Z\\
        Z^3&0
    \end{pmatrix}} \Ocal(1)^{\oplus 3} \oplus \Ocal(3) \rightarrow \underline{\Ecal}(3) \rightarrow 0
    $$
    where $\underline{\Ecal}$ is locally free. Furthermore, $\underline{\Ecal}$ lies in the following exact sequence 
    $$
    0\rightarrow \Ocal_{\PBbb^2} \rightarrow \underline{\Ecal} \rightarrow \underline{\Ical} \rightarrow 0
    $$
    where $\underline{\Ical}$ is locally isomorphic to $(x-y^2,y^3)$ which after a coordinate change is essentially isomorphic to $(x, y^3)$. In particular, as a locally complete intersection, $\underline{\Ical}$ can never be isomorphic to $\Ical$ which is not a locally complete intersection. 
\end{exam}

\begin{cor}\label{Different}
   In general, $\underline{\Ical}$ and $\Ical$ are different in nature. 
\end{cor}

Now we give a fertile family $\Ecal$ which forms $\Ocal_{B_0} \oplus (x^3, y^3)$ and the singularity of the bubble on the first level is no longer of the type $\Ocal \oplus \Ical$. 

\begin{exam}
Consider the following
    $$
    0\rightarrow \Ocal \xrightarrow{\begin{pmatrix}
    x^3\\
    y^3+z^4\\
    z^4(x^2+y^2)
    \end{pmatrix}} \Ocal^{\oplus 3} \rightarrow \Ecal \rightarrow 0.
    $$
The bubble of this family is given by 
    $$
    0\rightarrow \Ocal \xrightarrow{\begin{pmatrix}
    X^3\\
    Y^3\\
    Z^4(X^2+Y^2)
    \end{pmatrix}} \Ocal(3)^{\oplus 2} \oplus \Ocal(6) \rightarrow \underline{\Ecal} \rightarrow 0.
    $$    
Note $\underline{\Ecal}$ has a singularity at $[0,0,1]$. Writing in local coordinates, we know $\underline{\Ecal}$ is modeled on 
      $$
      0\rightarrow \Ocal \xrightarrow{\begin{pmatrix}
      x^3\\
      y^3\\
      x^2+y^2
      \end{pmatrix}} \Ocal^{\oplus 3} \rightarrow \underline{\Ecal} \rightarrow 0
    $$
    near $0\in B_0$. Then locally 
    $$\underline{\Ecal} \ncong \Ocal \oplus \underline{\Ical}$$
    for any $\underline{\Ical}$. Otherwise, $\underline{\Ical}$ must be a locally complete intersection so that 
    $$
    \underline{\Ical}=(x^3,y^3, x^2+y^2).
    $$
    which is impossible.
\end{exam}

\begin{cor}
    In general, when forming the singularity of the type $\Ocal \oplus \Ical$, the singularity formed on the bubble trees does not have to be of the type $\Ocal \oplus \Ical$.
\end{cor}

So far, we have only focused on studying families forming singularities of the type $\Ocal \oplus \Ical$. Of course, there exists fertile families forming other type of singularities in general. For this, we give a simple example forming the singularity $(x,y)\oplus (x,y)$ which carries a smooth bubble of charge $2$.

\begin{exam}
 The following fertile family forms the singularity $(x,y)\oplus (x,y)$
     $$
    0\rightarrow \Ocal^{\oplus 2} \xrightarrow{\begin{pmatrix}
        x&z\\
        y&z\\
        z&x\\
        0&y
    \end{pmatrix}} \Ocal^{\oplus 4} \rightarrow \Ecal \rightarrow 0
    $$
    where the bubble tree has height one with the bubble being stable. More precisely, the bubble is given by 
    $$
    0\rightarrow \Ocal^{\oplus 2} \xrightarrow{\begin{pmatrix}
        X&Z\\
        Y&Z\\
        Z&X\\
        0&Y
    \end{pmatrix}} \Ocal(1)^{\oplus 4} \rightarrow \underline{\Ecal}(2) \rightarrow 0.
    $$
\end{exam}

\begin{rmk}
    The notion of bubbles and techniques developed in this paper for dealing with singularity of the type $\Ocal \oplus \Ical$ crucially uses the special role of the $\Ocal$ factor. The complete picture for more general classes of singularities seems to depend a lot on the type of the singularities and the families. 
\end{rmk}

\bibliography{papers}

\providecommand{\MR}[1]{}
\providecommand{\bysame}{\leavevmode\hbox to3em{\hrulefill}\thinspace}
\providecommand{\MR}{\relax\ifhmode\unskip\space\fi MR }
\providecommand{\MRhref}[2]{%
  \href{http://www.ams.org/mathscinet-getitem?mr=#1}{#2}
}
\providecommand{\href}[2]{#2}
\begin{thebibliography}{1}

\bibitem{ADHM}
Michael~F Atiyah, Nigel~J Hitchin, Vladimir~Gershonovich Drinfeld, and Yu~I
  Manin, \emph{Construction of instantons}, Instantons In Gauge Theories
  (1994), 133--135.

\bibitem{ChenSun:20b}
Xuemiao Chen and Song Sun, \emph{Algebraic tangent cones of reflexive sheaves},
  Int. Math. Res. Not. IMRN (2020), no.~24, 10042--10063. \MR{4190396}

\bibitem{DonaldsonKronheimer:90}
S.~K. Donaldson and P.~B. Kronheimer, \emph{The geometry of four-manifolds},
  Oxford Mathematical Monographs, The Clarendon Press, Oxford University Press,
  New York, 1990, Oxford Science Publications. \MR{1079726}

\bibitem{Donaldson83}
Simon~K Donaldson, \emph{An application of gauge theory to four-dimensional
  topology}, Journal of Differential Geometry \textbf{18} (1983), no.~2,
  279--315.

\bibitem{Donaldson:84}
\bysame, \emph{Instantons and geometric invariant theory}, Communications in
  Mathematical Physics \textbf{93} (1984), 453--460.

\bibitem{OSS11}
Christian Okonek, Michael Schneider, and Heinz Spindler, \emph{Vector bundles
  on complex projective spaces: With an appendix by si gelfand}, Springer
  Science \& Business Media, 2011.

\bibitem{Taubes82}
Clifford~Henry Taubes, \emph{{S}elf-dual {Y}ang-{M}ills connections on
  non-self-dual 4-manifolds}, Journal of Differential Geometry \textbf{17}
  (1982), no.~1, 139--170.

\bibitem{Uhlenbeck:82a}
Karen~K. Uhlenbeck, \emph{Connections with {$L^{p}$} bounds on curvature},
  Comm. Math. Phys. \textbf{83} (1982), no.~1, 31--42. \MR{648356}

\bibitem{Uhlenbeck:82b}
\bysame, \emph{Removable singularities in {Y}ang-{M}ills fields}, Comm. Math.
  Phys. \textbf{83} (1982), no.~1, 11--29. \MR{648355}

\end{thebibliography}

\end{document}